\documentclass[]{amsart}

\usepackage{amssymb}           
\usepackage{bm}                

\usepackage{amsmath}           
\usepackage{amsthm}            
\usepackage{amscd}             

\usepackage{ascmac}            
\usepackage{indentfirst}       

\usepackage[dvipdfmx]{graphicx} 
\usepackage{comment}

\newtheorem{theorem}{Theorem}[section]
\newtheorem{prop}[theorem]{Proposition}
\newtheorem{cor}[theorem]{Corollary}
\newtheorem{lem}[theorem]{Lemma}

\theoremstyle{definition}

\newcommand{\C}{\ensuremath{\mathbb{C}}}

\newcommand{\Z}{\ensuremath{\mathbb{Z}}}
\newcommand{\N}{\ensuremath{\mathbb{N}}}

\begin{document}


%
%
\renewcommand{\subjclassname}{
\textup{2020} Mathematics Subject Classification}\subjclass[2020]{Primary 
57K12, 
20F34. 
Secondary 
57K45, 
20F12. 
}

\date{\today}
\keywords{Knot quandle; fibered knot; free group}


\title[The knot quandles of Suciu's ribbon $n$-knots]{
The knot quandles of 
Suciu's ribbon $n$-knots and 
automorphisms on 
the free group of rank two
}


\author{Takuya Sakasai}
\address{Graduate School of Mathematical Sciences, 
The University of Tokyo, 
3-8-1 Komaba, 
Meguro-ku, Tokyo, 153-8914, Japan}
\email{sakasai@ms.u-tokyo.ac.jp}

\author{Kokoro Tanaka}
\address{Department of Mathematics, 
Tokyo Gakugei University, 
4-1-1 Nukuikita-machi, 
Koganei-shi, Tokyo 184-8501, Japan}
\email{kotanaka@u-gakugei.ac.jp}



\begin{abstract}
Jab{\l}onowski proved that the knot quandles of Suciu's $n$-knots, which share isomorphic knot groups, are mutually non-isomorphic, and Yasuda later gave a different proof. In this paper, we present yet another proof of this result by analyzing the conjugacy classes of certain automorphisms of the free group of rank two.


\end{abstract}

\maketitle



\section{Introduction}
\label{sec:intro}

%
Suciu \cite{Suc85} constructed a family of fibered ribbon $n$-knots in $S^{n+2}$ for $n>1$, whose knot groups are all isomorphic to that of the trefoil. 
He distinguished these knots using the second homotopy groups of their complements. 
Later, Kanenobu and Sumi \cite{KaSu20} computed the fundamental groups of their double branched covers and showed that these groups are mutually non-isomorphic via $\operatorname{SL}(2,\C)$-representations, thereby distinguishing the knots.  
A \emph{quandle} \cite{Joy82,Mat82} is an abstract algebraic structure whose axioms correspond to the Reidemeister moves on $1$-knot diagrams. 
In knot theory, the most important example is the \emph{knot quandle}, which is associated with an oriented $n$-knot and serves as a knot invariant. 
It is closely related to the knot group, in that information about the knot group can be extracted from the knot quandle. 
The converse will be mentioned in the end of this paper.

Jab{\l}onowski \cite{Jab25-pre} proved that the knot quandles of Suciu's $n$-knots are mutually non-isomorphic by comparing their \emph{core group}s, and Yasuda \cite{Yas25-pre} later provided an alternative proof. More specifically, his alternative proof can be described as follows. Winker \cite{Win84} established for $n=1$ that the knot quandle of an oriented $n$-knot carries information about the fundamental group of its double branched cover. Recently, Yasuda \cite{Yas25-pre} extended this result to all $n \ge 1$. As a consequence, combining his result with the work of Kanenobu and Sumi \cite{KaSu20}, he reproved that the knot quandles of Suciu's $n$-knots are mutually non-isomorphic, and hence distinguish these knots. 
%
%
In this paper, we provide yet another proof of this result by analyzing conjugacy classes of certain automorphisms of the free group of rank two.  

Aside from non-isomorphism, a fundamental invariant of quandles is their \emph{type}, which measures their complexity and takes values in $\N \cup \{\infty\}$.
%
Yasuda \cite{Yas25-pre} also showed that the types of these quandles are infinite by analyzing linear representations of the $3$-strand braid group, which is isomorphic to the knot group of the trefoil.
We give another proof of this infiniteness by computing the orders of the above automorphisms of the free group of rank two. 
The authors believe our alternative proof offers an additional perspective on the diversity of Suciu's knots.

\section{Suciu's $n$-knots and the free group of rank two}
\label{sec:}

Suciu \cite{Suc85} constructed a family of fibered ribbon $n$-knots $R_k$ in $S^{n+2}$ for $n>1$, indexed by positive integers $k$.
The fiber of $R_k$ is a punctured $S^1 \times S^n \mathbin{\#} S^1 \times S^n$, whose fundamental group is isomorphic to the free group $F$ of rank two, generated by $a$ and $b$.
For each positive integer $k$, let $f_k$ be an element of the automorphism group $\operatorname{Aut}(F)$ of $F$ defined by $f_k(a) = a^k b a^{-k}$ and $f_k(b) = a^{k-1} b a^{-k}$.
We fix an isomorphism that sends the homotopy class of the loop $S^1 \times \{*\}$ in the left $S^1 \times S^n$ to $a$ and that in the right $S^1 \times S^n$ to $b$.
With respect to this identification, the automorphism induced by the monodromy coincides with $f_k$; see \cite[p.488]{Suc85}.


\begin{prop}
\label{prop:order}
The order of $f_k$ in $\operatorname{Aut}(F)$ is infinite for any positive integer $k$. 
\end{prop}

\begin{proof}
The abelianization map from $F$ to $\mathbb{Z}^2$ induces a natural homomorphism from $\operatorname{Aut}(F)$ to $\operatorname{GL}(2,\mathbb{Z})$, and the image of $f_k$ is $\begin{pmatrix}0 & -1 \\ 1 & 1\end{pmatrix}$ for all $k$. A simple calculation shows that this matrix has order $6$. Hence, if $f_k$ has finite order, it must be a multiple of $6$. However, any nontrivial periodic automorphism from $\operatorname{Aut}(F)$ has order $2$, $3$, or $4$; see \cite{Mes74}. Therefore, the order of $f_k$ is infinite. 
\end{proof}

From the above argument, the map $f_k^6$ lies in the kernel of the natural homomorphism from $\operatorname{Aut}(F)$ to $\operatorname{GL}(2,\mathbb{Z})$. Since $\operatorname{GL}(2,\mathbb{Z})$ is isomorphic to the outer automorphism group $\operatorname{Out}(F)$ of the rank-two free group $F$, it follows that $f_k^6$ belongs to the inner automorphism group $\operatorname{Inn}(F)$ of $F$.  
Since the center of $F$ is trivial, 
the natural homomorphism 
$I \colon F \to \operatorname{Inn}(F)$
is an isomorphism. See \cite[Chapter I]{LS-book}, for example. 
%
At this point, it is natural to attempt a direct computation of $f_k^6$, which yields the following.

\begin{lem}
The map $f_k^6$ is the inner automorphism $I(x_k)$ of $F$ with 
\[
x_k:=a^k b^k (a^{-1} b)^{k-1} a^{-k} b^{-k} (a b^{-1})^{k-1},
\]
which lies in the commutator subgroup $[F,F]$ of $F$. 
\end{lem}

\begin{proof}
Let $f$ be the automorphism of $F$ defined by $f(a) = b$ and $f(b) = a^{-1}b$. 
Then for each positive integer $k$, we have $f_k = I(a^k) \circ f$. 
%
By induction, for any positive integer $\ell$, 
we also have $f_k^{\ell+1} = I(a^k f(a)^k \cdots f^{\ell}(a)^k) \circ f^{\ell+1}$. 
%
Moreover, by computing up to $f^6(a)$ and $f^6(b)=f^7(a)$, one sees that
$f^6 = I(b^{-1}ab a^{-1})$, 
as can be verified from the following calculations:
\[
\begin{aligned}
& f(a)=b, \quad f^2(a)=a^{-1}b, \quad f^3(a)=b^{-1}a^{-1}b, \\[1mm]
& f^4(a)=b^{-1}ab^{-1}a^{-1}b=(b^{-1}a)b^{-1}(b^{-1}a)^{-1}, \\[1mm]
& f^5(a)=b^{-1}a^2 b^{-1} a^{-1}b=(b^{-1}ab)b^{-1}a(b^{-1}ab)^{-1}, \\[1mm]
& f^6(a)=b^{-1}ab\,a b^{-1}a^{-1}b = (b^{-1}ab a^{-1})\,a\,(b^{-1}ab a^{-1})^{-1}, \\[1mm]
& f^7(a)=b^{-1}ab a^{-1} b a b^{-1} a^{-1} b=(b^{-1}ab a^{-1})\,b\,(b^{-1}ab a^{-1})^{-1}.
\end{aligned}
\]
Hence, we have $f_k^6 = I\bigl(a^k f(a)^k \cdots f^5(a)^k \,(b^{-1}ab a^{-1})\bigr)$.  
A straightforward computation shows that
\[
a^k f(a)^k \cdots f^5(a)^k \,(b^{-1}ab a^{-1}) 
= a^k b^k (a^{-1}b)^{\,k-1} a^{-k} b^{-k} (ab^{-1})^{\,k-1}. 
\] 
Since the total exponents of both $a$ and $b$ in this element are zero, it lies in the commutator subgroup $[F,F]$.
\end{proof}


\begin{prop}
\label{prop:non-conj}
The automorphisms $f_k$ are mutually non-conjugate in $\operatorname{Aut}(F)$. 
\end{prop}

\begin{proof}
It suffices to show that the elements $f_k^6$ in $\operatorname{Inn}(F)$ are mutually non-conjugate in $\operatorname{Aut}(F)$. 
Since $f \circ I(x) \circ f^{-1} = I(f(x))$ 
for any $x \in F$ and $f \in \operatorname{Aut}(F)$, 
the conjugacy class of $f_k^6 = I(x_k)$ in $\operatorname{Aut}(F)$ corresponds to the $\operatorname{Aut}(F)$-orbit of $x_k$ in $F$. 
Hence, it is enough to show that the $\operatorname{Aut}(F)$-orbits of the elements $x_k$ in $[F,F]$ are distinct. 
%
We then focus on the $\operatorname{Aut}(F)$-equivariant map from $[F,F]$ to $[F,F]/[[F,F],F]$ and show that the $\operatorname{Aut}(F)$-orbits of the images of $x_k$ are distinct. 
Note that the codomain is the center of the $2$-step nilpotent quotient of $F$, whose structure is well understood.
For instance, it is isomorphic to $\wedge^2 \Z^2 \cong \Z$; see \cite{MKS-book}, 
where we use that the abelianization of $F$ is $\Z^2$, 
with the equivalence class of the commutator $[a,b]$ as a generator.
Moreover, the action of $\operatorname{Aut}(F)$ on the codomain factors through $\mathrm{Aut}(F) \twoheadrightarrow \mathrm{GL}(2,\mathbb{Z}) \xrightarrow{\text{det}} \{\pm 1\}$, i.e., it is given by the sign of the determinant.

First, we rewrite $x_k$ in $[F,F]$ modulo $[[F,F],F]$ as a product of commutators:
\[
\begin{aligned}
x_k
&= 
a^k b^k (a^{-1} b)^{k-1} a^{-k} b^{-k} (a b^{-1})^{k-1} \\
&= 
a^k b^k (a^{-1} b)^{k-1} \cdot [a^{-k}, b^{-k}] b^{-k} a^{-k} \cdot ([a, b^{-1}] b^{-1} a)^{k-1} \\
&\equiv 
a^k b^k (a^{-1} b)^{k-1} b^{-k} a^{-k} (b^{-1} a)^{k-1} \cdot [a^{-k}, b^{-k}] [a, b^{-1}]^{k-1} \\
&= 
[a^k b^k, (a^{-1} b)^{k-1}] [a^{-k}, b^{-k}] [a, b^{-1}]^{k-1}
\end{aligned}
\]

Next, we express each commutator in terms of the generator $[a,b]$.
To do so, we use the isomorphism $[F,F]/[[F,F],F] \cong \wedge^2 \Z^2 \cong \Z$.
Namely, we apply the abelianization map from $F$ into $\Z^2$ to convert products into sums,
and replace each commutator with the exterior product. The computations then yield:
\[
[a^k b^k, (a^{-1} b)^{k-1}] 
\;\mapsto\;
(ka+kb) \wedge (-(k-1)a+(k-1)b) 
\;=\; 2k(k-1)\, a \wedge b
\]
\[
[a^{-k}, b^{-k}] 
\;\mapsto\;
(-ka) \wedge (-kb) 
\;=\; k^2 \, a \wedge b
\]
\[
[a, b^{-1}]^{k-1} 
\;\mapsto\;
(k-1) (a \wedge (-b) )
\;=\; -(k-1)\, a \wedge b
\]
In summary, we obtain
$x_k \equiv [a,b]^{2k(k-1)} \cdot [a,b]^{k^2} \cdot [a,b]^{-(k-1)} = [a,b]^{3k^2-3k+1}$.
Since $3k^2-3k+1 > 0$ is strictly increasing for $k \ge 1$,
these are distinct integers up to sign.
Hence, we see that the $\operatorname{Aut}(F)$-orbits of the images of $x_k$ are distinct.
\end{proof}

\section{The knot quandles of Suciu's $n$-knots}
\label{sec:}
%
A typical example of quandles is the \emph{generalized Alexander quandle} $\operatorname{GAlex}(G,\varphi)$, which is associated with a group $G$ and an automorphism $\varphi$ in $\operatorname{Aut}(G)$.  
%
%
Inoue~\cite{Ino19} proved that if $K$ is a fibered oriented $n$-knot for $n>1$, then its knot quandle $Q(K)$ 
is isomorphic to a generalized Alexander quandle $\operatorname{GAlex}(G,\varphi)$, 
where $G$ is the fundamental group of the fiber of $K$, 
and $\varphi$ is the automorphism of $G$ induced by the monodromy of the complement of $K$.  
From this, it follows in particular that the knot quandle $Q(R_k)$ of Suciu's fibered ribbon $n$-knot $R_k$ is isomorphic to $\operatorname{GAlex}(F,f_k)$.
With these preliminaries in place, we provide alternative proofs of 
\cite[Theorem~5.1]{Jab25-pre} and \cite[Theorem~1.3]{Yas25-pre}, in that order. 
We remark that \cite[Theorem~5.1]{Jab25-pre} and \cite[Theorem~1.2]{Yas25-pre} state the same result, as mentioned in the introduction.

\begin{theorem}
\label{thm:non-iso}
The knot quandles $Q(R_k)$ are mutually non-isomorphic. 
\end{theorem}

\begin{proof}
When $\operatorname{GAlex}(G,\psi)$ and $\operatorname{GAlex}(G,\varphi)$ are \emph{connected} (in the standard terminology of quandle theory), they are isomorphic if and only if $\psi$ and $\varphi$ are conjugate in $\operatorname{Aut}(G)$; see \cite[Lemma~B.1]{CDS17} and \cite{HKKK24-pre}. 
Since knot quandle $Q(R_k)$ is connected \cite[Lemma~2.27]{Nos17-book} and is isomorphic to $\operatorname{GAlex}(F,f_k)$, Proposition~\ref{prop:non-conj} immediately implies the statement of the theorem. 
\end{proof}

\begin{theorem}
\label{thm:type}
The type of $Q(R_k)$ 
is infinite for all positive integers $k$. 
\end{theorem}

\begin{proof}
Since $Q(R_k)$ is isomorphic to $\operatorname{GAlex}(F,f_k)$, it is enough to determine the type of the latter. It has been shown in \cite[Proposition~2.1]{TaTa23-pre} that the type of a generalized Alexander quandle coincides with the order of its defining automorphism. Therefore, Proposition~\ref{prop:order} implies that the type of $Q(R_k)$ is infinite. 
\end{proof}


We have mentioned that the knot quandle carries information about the knot group.
As also noted in \cite{Yas25-pre}, we make here a remark concerning the converse. 
The second author, together with Taniguchi \cite{TaTa23-pre}, exhibited a triple of oriented $2$-knots whose knot groups are isomorphic, but whose knot quandles are mutually non-isomorphic. 
In that example, the types of the knot quandles differ, and this difference distinguishes the knot quandles. 
Although it is not explicitly stated in their paper, we note that by applying Artin's spinning construction \cite{Art25}, one can construct such triples of oriented $n$-knots for any $n>1$. 
This naturally raises the question of whether examples exist in which the types also coincide. 
In fact, combining the two theorems shows that Suciu's ribbon $n$-knots provide such examples, as stated in the following corollary: 

\begin{cor}
There exists an infinite family of oriented $n$-knots for $n>1$ whose knot groups are mutually isomorphic, whose knot quandles have infinite type, but whose knot quandles are mutually non-isomorphic.
\end{cor}



\section*{Acknowledgments}
The authors thank Michal Jablonowski 
for bringing his preprint \cite{Jab25-pre} to their attention.
The first-named author has been supported in part by 
KAKENHI 
(No.~23K20799, No.~24K06740), 
Japan Society for the Promotion of Science. 
The second-named author has been supported in part by 
KAKENHI 
(No.~21K03220, No.~25K06998), 
Japan Society for the Promotion of Science.


\bibliographystyle{amsplain}
\bibliography{reference-suciu}
\end{document}